\documentclass{amsart}

\usepackage[dvips]{color}
\usepackage{amsfonts}
\usepackage{multicol}
\usepackage{graphicx}
\usepackage{amsmath}
\usepackage{amssymb}
\usepackage{amsthm}
\usepackage{mathrsfs}
\newtheorem{theorem}{Theorem}[section]
\newtheorem{lemma}[theorem]{Lemma}

\newtheorem{corollary}[theorem]{Corollary}
\theoremstyle{definition}

\theoremstyle{remark}
\newtheorem{remark}[theorem]{Remark}

\numberwithin{equation}{section}

\newfam\msbfam
\font\tenmsb=msbm10  \textfont\msbfam=\tenmsb
\font\sevenmsb=msbm7  \scriptfont\msbfam=\sevenmsb
\font\fivemsb=msbm5    \scriptscriptfont\msbfam=\fivemsb
\def\Bbb{\fam\msbfam \tenmsb}

\def\CC{{\Bbb C}}

\newfam\msbbfam
\font\tenmsbb=msbm10  scaled \magstep1 \textfont\msbbfam=\tenmsbb
\font\sevenmsbb=msbm7  scaled \magstep1 \scriptfont\msbbfam=\sevenmsbb
\font\fivemsbb=msbm5    scaled \magstep1 \scriptscriptfont\msbbfam=\fivemsbb
\def\BBbb{\fam\msbbfam \tenmsbb}

\def\CCC{{\BBbb C}}

\begin{document}

\title{Approach regions for domains in $\CCC^2$ of finite type}

\author{Baili Min}
\address{Department of Mathematics, Washington University in St.\,Louis, Saint Louis, MO 63130}
\email{minbaili@math.wustl.edu}


%    General info

\date{2010}


\begin{abstract} Recall the Fatou theorem for the unit disc in $\CC$.
Consider a domain in $\mathbb{C}^2$ of finite type. In this paper we will show that the approach regions studied by Nagel, Stein, Wainger and Neff are the best possible ones for the boundary behavior of bounded analytic functions, and there is no Fatou theorem for complex tangentially broader approach regions.

\end{abstract}

\maketitle

\section{Background}
The purpose of this paper is to study one kind of approach region for domains in $\mathbb{C}^2$ of finite type.
Have a look at the space $\CC$.

\subsection{The problem}
If we consider the unit disc in $\mathbb{C}$, the classical theorem of Fatou states that, for $f \in H^p$, the nontangential limit exists for almost every boundary point. In the quest for generalization to the case of several complex variables,  Kor\'{a}nyi discovered the ``admissible'' approach region which allows parabolical approach from certain directions, in the case of balls and bounded functions (see \cite{K}). This phenomenon was generalized by Stein, who defined the admissible approach region (see \cite{S}) for holomorphic functions in $H^p$ in strongly pseudoconvex domains in $\mathbb{C}^n$. It was shown by Hakim and Sibony that, on the unit ball in $\mathbb{C}^n$, this is the best possible approach region (see \cite{HS}).  For meromorphic functions, but in the Nevanlinna class, Lempert defined another approach region for pseudoconvex domains in $\mathbb{C}^n$ (see \cite{Lempert}).

Kohn introduced the concept of finite type when he studied the $\overline{\partial}$ problem (see \cite{Kohn}), which eventually had great impact on the geometry of hypersurfaces in $\mathbb{C}^n$. Among all the work on the finite type conditions, we should mention that by Kohn, Bloom/Graham Catlin and D'Angelo: they studied the matter in terms of ideals and iterated commutators (see \cite{Kohn}, \cite{BG}, \cite{C},  \cite{D1} and \cite{D2}). 

We now wish to know what approach regions will be like if the domain is of finite type, which means it may not be strongly pseudoconvex. Nagel, Stein and Wainger defined  the admissible approach region (see \cite{NSW}) for holomorphic functions and then, in his Ph.\,D.\;dissertation (see \cite{Neff}) Neff showed that the approach region would also work for meromorphic Nevanlinna functions.

So far, however, it is not known whether these approach regions for the finite type case are the best possible. In this paper we will study the approach region of Nagel-Stein-Wainger-Neff type in $\mathbb{C}^2$. The main results are Theorem~\ref{main} and Corollary~\ref{main2} which assure us that, for other regions  broader only in the tangential direction, we can construct a bounded holomorphic function that does not have a limit at the base-points, and consequently there is no Fatou theorem for those broader approach regions: these base-points form a set of positive measure on the boundary. For the rest of this paper, everything is carried out in $\mathbb{C}^2$.

To start with, we will approach the concept of finite type by iterated commutators. 

\subsection{A study of iterated commutators}

Let $\Omega$ be a smoothly bounded domain in $\mathbb{C}^2$ with a defining function $\rho$. We assume that $\Omega$ is of finite type. Suppose $\omega^0=(\omega_1^0, \omega_2^0) \in \partial \Omega$. Then, in a small neighborhood $V=V_{\omega^0}$ of $\omega^0$, the complex holomorphic tangential vector field has a basis $L+i\overline{L}$, where
\begin{equation}
L=-\frac{\partial \rho}{\partial z_2}\frac{\partial}{\partial z_1}+\frac{\partial \rho}{\partial z_1}\frac{\partial}{\partial z_2}
\label{eqL},
\end{equation}
and
\begin{equation}
\overline{L}=-\frac{\partial \rho}{\partial \overline{z}_2}\frac{\partial}{\partial \overline{z}_1}+\frac{\partial \rho}{\partial \overline{z}_1}\frac{\partial}{\partial \overline{z}_2}
\label{eqLbar}.
\end{equation}

Then we can find a transverse vector field $T$ such that $L, \overline{L}$ and $T$ span the tangent space to $\partial \Omega$ at any point in $V$:
\begin{equation}
T=\frac{\partial \rho}{\partial \overline{z}_1}\frac{\partial}{\partial z_1}-\frac{\partial \rho}{\partial z_1}\frac{\partial}{\partial \overline{z}_1}
\label{eqT}.
\end{equation}

Suppose that $\mathscr{L}_{k-1}$ is an iterated commutator of degree $k-1$, 
\begin{equation}
\mathscr{L}_{k-1}=f_1L+f_2\overline{L}+\lambda_{k-1}T,
\label{eqLk-1}
\end{equation}
or simply $\mathscr{L}_{k-1} \equiv \lambda_{k-1}T \ \text{mod}(L, \overline{L})$.

If $\mathscr{L}_k=[L,\mathscr{L}_{k-1}]$, we can compute that $\mathscr{L}_k \equiv \lambda_k T \ \text{mod}(L, \overline{L})$ where $\lambda_k$ can be expressed explicitly:
\begin{align}
\lambda_k&=\frac{\partial \lambda_{k-1}}{\partial z_2}\frac{\partial \rho}{\partial z_1} - \frac{\partial \lambda_{k-1}}{\partial z_1}\frac{\partial \rho}{\partial z_2}+ \lambda_{k-1}\frac{\partial^2 \rho}{\partial z_1 \partial z_2 }- \lambda_{k-1}\frac{\frac{\partial \rho}{\partial z_2}\frac{\partial^2 \rho}{\partial z_1^2}}{\frac{\partial \rho}{\partial z_1}} \nonumber \\
&\qquad+\frac{f_2}{\frac{\partial \rho}{\partial z_1}\frac{\partial \rho}{\partial \overline{z}_1}}\Big(\frac{\partial^2 \rho}{\partial z_1 \partial \overline{z}_1}\frac{\partial \rho}{\partial z_2}\frac{\partial \rho}{\partial \overline{z}_2}+ 
\frac{\partial^2 \rho}{\partial z_2 \partial \overline{z}_2}\frac{\partial \rho}{\partial z_1}\frac{\partial \rho}{\partial \overline{z}_1} \nonumber \\
&\qquad \qquad \qquad \qquad-\frac{\partial^2 \rho}{\partial z_1 \partial \overline{z}_2}\frac{\partial \rho}{\partial z_2}\frac{\partial \rho}{\partial \overline{z}_1}-
\frac{\partial^2 \rho}{\partial z_2 \partial \overline{z}_1}\frac{\partial \rho}{\partial z_1}\frac{\partial \rho}{\partial \overline{z}_2}
        \Big).
\label{eqlambda}
\end{align}

We can get similar results for $[\overline{L},\mathscr{L}_{k-1}]$. This computation shows that, for any iterated commutator of degree $k$, only $\frac{\partial \rho}{\partial z_1}$ and/or $\frac{\partial \rho}{\partial \overline{z}_1}$ appear in the denominator of the coefficient function of the complex normal vector $T$. So we choose coordinates so that
the $z_1$ derivatives of $\rho$ do not vanish.

Let $\mathscr{M}_k$ be the collection of all these linearly independent iterated commutators with degree less or equal to $k$. Suppose $\mathscr{L} \in \mathscr{M}_k$, and that $\lambda_{\mathscr{L}}$ is the coefficient function of $T$ in the sense that $\mathscr{L} \equiv \lambda_{\mathscr{L}}T \ \text{mod}(L, \overline{L})$. Then we can define $\Lambda_k(z)$ by:
\begin{equation}
\Lambda_k(z)=\sqrt{\sum_{\mathscr{L} \in \mathscr{M}_k}{\lambda_{\mathscr{L}}^2(z)}},
\end{equation}
a key function for defining the approach regions.

\begin{remark}
If $\Lambda_{k-1}(z^0) \neq 0$, then  $\Lambda_k(z^0) \neq 0$. 
Actually, the smallest $\tau$ such that $\Lambda_\tau(z^0) \neq 0$ is called the {\it type} of $z^0$. See \cite{D2} and \cite{SK}.
\end{remark}

\begin{remark}
Note that we always have $\nabla \rho \neq 0$, since $\rho$ is a defining function. With the assumption that $\frac{\partial \rho}{\partial z_1}(z^0) \neq 0$ and $\frac{\partial \rho}{\partial \overline{z}_1}(z^0) \neq 0$, all $\Lambda_k(z^0) < \infty$, $k \geqslant 2$.
\end{remark}

\section{Approach Regions}

\subsection{Definitions}
Let $\Omega$ be a domain of finite type in $\mathbb{C}^2$ such that, for all $z \in \Omega$, it is true that $|z| \leqslant 1$. Suppose that $(1,0)$ is on the boundary,  that $\frac{\partial \rho}{\partial z_1}$ and $\frac{\partial \rho}{\partial \overline{z}_1}$ do not vanish at $(1,0)$, and the vector $\langle 1,0 \rangle$ is also a outward normal vector to the boundary at $(1,0)$. Let $U \subset \partial \Omega$ be a neighborhood of $(1,0)$ small enough that for any $w=(w_1, w_2) \in U$, the vector $(1,0)$ is transversal to $U$ at $w$. 

Let $\tau_z$ be the type of the point $z$ if $z \in \partial \Omega$, or the type of $\pi(z)$,the Euclidean normal projection of $z$ on the boundary. We also denote the ordinary Euclidean distance of $z$ to $\partial \Omega$  by $\delta(z)=\big|z-\pi(z)\big|$. Let $\tau=\max_{z \in \partial \Omega}\tau_z$. Since we assume that $\Omega$ is of finite type, we must have $\tau < \infty$. We denote $\tilde{\tau}$ the type of $(1,0)$. Then of course $\tilde{\tau} \leqslant \tau$.

Define $D(z)$:
\begin{equation}
D(z)=\inf_{2 \leqslant k \leqslant \tau}\Big(\frac{\delta(z)}{\Lambda_k\big(\pi(z)\big)}\Big)^{1/k}.
\label{equD}
\end{equation}

Define the ball $\beta_2$ such that, for $\omega^0 \in \partial \Omega$ and $r>0$, $\omega \in \beta_2(\omega^0, r)$ if and only if $\omega \in \partial \Omega$ and 
\begin{equation}
\left\{
\begin{array}{lr}
|\omega-\omega^0|< r,&\\
\Big|R\big(\omega,\omega^0\big)\Big|<\Lambda^{r}(\omega^0),& \\
\end{array} \right.
\label{ball}
\end{equation}
where we use this notation:
\begin{equation}
\Lambda^{\theta}(\zeta)=\sum_{k=2}^{\tau}{\theta^k\Lambda_k(\zeta)},
\label{equLambda}
\end{equation}
and where $R$ is a polarization of $\rho$, that is, $R(z, w)$ is a $C^\infty$ complex-valued function satisfying the following requirements:
\begin{align}
 &R(z,z)=\rho(z),\\
 &\overline{\partial}_zR(z,w) \text{ vanishes to infinite order on } z=w,\\
 &R(z,w)-\overline{R(w,z)}\text{ vanishes to infinite order on } z=w.
\end{align}
For example, if $\rho(z)=z_1 \overline{z}_1+z_2^2 \overline{z}_1^2-1$ is a defining function for a domain in $\CC^2$, then we can choose one polarization $R(z, w)=z_1 \overline{w}_1+z_2^2 \overline{w}_1^2-1$.

With these notations, the approach region of Nagel-Stein-Wainger-Neff type is
\begin{equation}
\mathscr{A}_{\alpha}(1,0)=\big\{ z\in \Omega \cap V: \pi(z) \in \beta_2\big((1,0), \alpha D(z) \big)\big\},
\label{ardef}
\end{equation}
where $\alpha > 0$. These definitions can be found in \cite{NSW} and \cite{Neff}.

We see that the definition of the approach region above is equivalent to
\begin{equation}
\left\{
\begin{array}{lr}
|\pi(z)-(1,0)|< \alpha D(z),&\\
\Big|R\big(\pi(z),(1,0)\big)\Big|<\Lambda^{\alpha D(z)}(1,0).& \\
\end{array} \right.
\label{arequiv}
\end{equation}

\begin{lemma}
If $\big|\pi(z)-(1,0)\big| \sim D(z)$, then $|z-(1,0)| \sim  D(z)$.
\end{lemma}

\begin{proof}
From the definition of $D(z)$ and the discussion of the iterated commutators, we know that 
\begin{equation}
D(z)=\Big(\frac{\delta(z)}{\Lambda_{\tau_z}\big(\pi(z)\big)}\Big)^{1/{\tau_z}} \sim \big(\delta(z)\big)^{1/{\tau_z}}, \text{or } \big(D(z)\big)^{\tau_z} \sim \delta(z),
\label{estimate of D}
\end{equation}
where $\tau_z=\tau\big(\pi(z)\big)$ is the type of $\pi(z)$.

We know that, since $\tau_z \geqslant 2$, as $\delta(z) \ll  1$, it is true that $\delta(z) \ll D(z)$.

As a result,
\begin{equation}
|z-(1,0)|=|z-\pi(z)+\pi(z)-(1,0)| \sim D(z).
\label{lemmapf}
\end{equation}
\end{proof}

Therefore we know that the following defines an approach region, denoted by $\mathscr{A}(1,0)$, which is comparable to $\mathscr{A}_1(1,0)$:
\begin{equation}
\left\{
\begin{array}{lr}
|z-(1,0)|<D(z),&\\
\Big|R\big(\pi(z),(1,0)\big)\Big|<\Lambda^{D(z)}(1,0).& \\
\end{array} \right.
\label{armodify}
\end{equation}

\section{The Best Approach Region}
Let $h_1$ and $h_2$ be two real-valued continuously decreasing functions such that $h_i: (0,1] \to [1, +\infty)$ and $\lim_{x \to 0+}{h_i(x)}=+\infty$, $i=1,2$. We may assume that they decrease to 1 very slowly. 

Now we consider an approach region in $\Omega$ at the point $w \in \partial \Omega$, denoted by $\mathscr{A}_{h_1, h_2}(w)$, defined by the following inequalities:
\begin{equation}
\label{complexbroader}
\left\{
\begin{array}{lr}
|z-w|< h_1\big(\delta_n(z)\big) D(z),&\\
\Big|R\big(\pi(z),w\big)\Big|<h_2\big(\delta_n(z)\big)\Lambda^{D(z)}(w),& \\
\end{array} \right.
\end{equation}
where $\delta_n(z)$ is the distance from $z$ to $\pi(z)$ in the complex normal direction.

We can compare $\mathscr{A}_{h_1, h_2}(1,0)$ with $\mathscr{A}(1,0)$ to see how these two kinds of domain are related.

First of all, $\mathscr{A}(1,0) \subseteq  \mathscr{A}_{h_1, h_2}(1,0)$. If $z \in \mathscr{A}_{h_1, h_2}(1,0)-\mathscr{A}(1,0)$, then $\delta_n(z)$ is very small. This means $\mathscr{A}_{h_1, h_2}(1,0)$ is very similar to $\mathscr{A}(1,0)$, but compared with $\mathscr{A}(1,0)$ it is broader in the complex tangential direction.

Then, the main result of this paper is: there is no Fatou's theorem for this kind of tangentially broader region $\mathscr{A}_{h_1, h_2}$. Therefore the approach regions of Nagel-Stein-Wainger-Neff type are the best possible ones. 

To see this, we are going to construct a bounded holomorphic function $f$ that does not have a limit $\mathscr{A}_{h_1, h_2}$-admissibly at any point in $U$. It is inspired by Hakim and Sibony's work in \cite{HS}.

For each $r>0$, there exists a set of points $\{\zeta_j\}_{j \in J}$, such that $\{\beta_2(\zeta_j, r^\tau)\}$ is a maximal family of pairwisely disjoint balls in $U$(See \cite{HS}). 

For each $\zeta_j=(\zeta_{j,1},\zeta_{j,2})$, define 
\begin{displaymath}
V_r(\zeta_j)=\{\zeta \in U: |\zeta - \zeta_j| < K r^\tau, \big|R(\zeta, \zeta_j)\big| < \Lambda^{Kr}(\zeta_j)\},
\end{displaymath}
where $K$ is a positive constant.

We want to show that:
\begin{lemma}
\begin{displaymath}
\bigcup_{j \in J}V_r(\zeta_j)=U.
\end{displaymath}
\end{lemma}

\begin{proof}
First of all, we realize that we only need to prove that $U \subset \bigcup_{j \in J}V_r(\zeta_j)$.

Without loss of generality, we just need to show that there exists $\zeta_j$ such that 
\begin{equation}
(1,0) \in V_r(\zeta_j),
\label{bndcover}
\end{equation}
because, for any other point in the domain $U$, the same method below shows that it also belongs to $V_r(\zeta_i)$ for some $i \in J$.

Therefore we just need to check that
\begin{equation}
\left\{
\begin{array}{lr}
|(1,0)-\zeta_j|< K r^\tau,&\\
\Big|R\big((1,0),\zeta_j \big)\Big|<\Lambda^{Kr}(\zeta_j).& \\
\end{array} \right.
\label{bndcover2}
\end{equation}

Of course, points close enough to $(1,0)$ will satisfy \eqref{bndcover2}. Since we are considering a small neighborhood around $(1,0)$, we can find a point $w \in U$ such that $w$ satisfies the inequalities in \eqref{bndcover2} and we may assume that
\begin{equation}
\left\{
\begin{array}{lr}
|(1,0)-w|<(K-2) r^\tau,&\\
\Big|R\big((1,0),w\big)\Big|< r^\tau.& \\
\end{array} \right.
\label{w_equ}
\end{equation}

Since $\{\beta_2(\zeta_j, r^\tau)\}$ makes a maximal family in $U$, there must exist a point $\zeta_j$ in the ball $\beta_2(w,2r^\tau)$ for some $j \in J$. We then want to check that this $\zeta_j$ makes the inequalities in \eqref{bndcover2} valid, and then the claim is proved.

To see this, we first check an arbitrary point $\zeta \in \beta_2(w,2r^\tau)$. Immediately by the triangle inequality we know that 
\begin{align}
|1-\zeta|&\leqslant \big|(1,0)-w\big|+|w-\zeta| \nonumber \\
         &<(K-2)r^\tau+2r^\tau \nonumber \\
         &=K r^\tau .
         \label{key1}
\end{align}

To check the second inequality in \eqref{bndcover2}, 
we first have 
\begin{equation}
\Big|R\big((1,0),\zeta \big)\Big| \leqslant \Big|R\big((1,0),w \big)\Big|+ \Big|R\big((1,0),w \big)-R\big((1,0),\zeta \big)\Big|.
\end{equation}

Since we already know that $\Big|R\big((1,0),w \big)\Big| < r^\tau$ and 
\begin{equation}
\Big|R\big((1,0),w \big)-R\big((1,0),\zeta \big)\Big| < K_1 |w-\zeta| <2K_1 r^\tau,
\end{equation}
it is true that
\begin{equation}
\Big|R\big((1,0),\zeta \big)\Big| < K_2 r^\tau.
\label{key2}
\end{equation}

By inequalities \eqref{key1} and \eqref{key2}, we can choose a positive constant $K$ big enough such that 
\begin{equation}
\left\{
\begin{array}{lr}
|(1,0)-\zeta|<Kr^\tau,&\\
\Big|R\big((1,0),\zeta \big)\Big|< \Lambda^{Kr}(\zeta)=\Lambda_{\tau_\zeta}(\zeta)(Kr)^{\tau_\zeta}+\cdots+\Lambda_{\tau}(\zeta)(Kr)^{\tau} .& \\
\end{array} \right.
\end{equation}

Since there must be one $\zeta_j$ in $\beta_2(w, 2r^\tau)$ as argued, this $\zeta_j$ then satisfies the inequalities in \eqref{bndcover2}, which means that we have $(1,0) \in V_r(\zeta_j)$, and then our claim is proved.

\end{proof}

For $n \in \mathbb{N}$, $r>0$ and $\{\zeta_j\}_{j \in J} \subset U$, define
\begin{displaymath}
g_{n,r}(z)=\sum_{j \in J}\Big(\frac{r^\tau}{R(z, \zeta_j) -r^\tau}\Big)^{2n},
\end{displaymath} 
and then define $f_n=1-\varepsilon_n-g_{n,r}$, where $\varepsilon_n =n ^{-1/4}$, and we know that there exists a subsequence $\{\varepsilon_{n_k}\}$ with $\sum{\varepsilon_{n_k}} < \infty$.

\begin{lemma}
For any $z \in U$ and $n \in \mathbb{N}$ large enough, $|g_{n,r}(z)| \leqslant 1+ \frac{A}{n}$, where $A$ is a positive constant.
\end{lemma} 
\begin{proof}
Let $\zeta_0 \in U$ be an arbitrary point and $N_{k,r}$ be the number of balls $\beta_2(\zeta_j, r^\tau)$ that are contained in the ball $\beta_2(\zeta_0, kr^\tau)$. Then we know that
\begin{equation}
N_{k,r} \leqslant C k^t,
\end{equation}
where $t$ is a positive integer and $C$ is a positive constant.

Now fix a point $\zeta \in \Omega$. For any $k \in \mathbb{N}$, define a subfamily of $\{\zeta_j\}_{j \in J}$:
\begin{displaymath}
J(\zeta, k)=\{\zeta_j : kr^\tau \leqslant |R(\zeta, \zeta_j)| < (k+1) r^\tau\}.
\end{displaymath}

With these preparations, we can estimate $|g_{n,r}|$.

First of all, we notice that if $|R(z, \zeta_j)| \geqslant kr^\tau$, we can get
\begin{align}
\Big|\frac{r^\tau}{R(z,\zeta_j)-r^\tau}\Big|^{2n}&\leqslant \Big|\frac{r^\tau}{\big(|R(z,\zeta_j)\big|^2+r^{2\tau}\big)^{\frac{1}{2}}}\Big|^{2n} \nonumber \\
         &\leqslant \Big|\frac{r^\tau}{\big(k^2 r^{2\tau}+r^{2\tau}\big)^\frac{1}{2}}\Big|^{2n} \nonumber \\
         &=\frac{1}{(1+k^2)^n}.
\end{align}
It then follows that, for $n$ large enough,
\begin{equation}
|g_{n,r}(z)| \leqslant 1 + A_1\sum_{k=1}^{\infty}\frac{k^t}{(1+k^2)^n} \leq 1+\frac{A}{n}.
\end{equation}
\end{proof}

Then we will be able to see more about the functions $f_n$.

\begin{lemma}
For each $\zeta_j$, there exists a zero of $f_n$. Moreover, this zero will approach to $\zeta_j$ as $n$ goes to infinity.
\end{lemma}
\begin{proof}
Here we are just going to consider the case for $\zeta_1=(1,0)$. This method also applies for other $\zeta_j$.

We introduce two auxiliary functions:
\begin{displaymath}
\phi_n(z_1)=f_n(z_1,0)
\end{displaymath}
and
\begin{displaymath}
\psi_n(z_1)=1-\varepsilon_n-\Big(\frac{r^\tau}{R\big((z_1,0), (1,0)\big)-r^\tau}\Big)^{2n}.
\end{displaymath}

Immediately we know that if $R\big((z_1,0),(1,0)\big)=r^\tau\big(1-(1-\varepsilon_n)^{-\frac{1}{2n}}\big)$, then $z_1$ is a zero of $\psi$.

We can choose a positive sequence $\{\gamma_n\}$ with $\gamma_n=n^{-4/3}$. On the closed curve $z$ such that $R\big((z,0), (1,0)\big)-R\big((z_1,0), (1,0)\big)=\gamma_n r^\tau e^{i\theta}$, we estimate that 
\begin{equation}
|\psi_n(z)|=|2n\gamma_n e^{i\theta}(1-\varepsilon_n)^{1/{2n}}+ O(\varepsilon^2)|.
\end{equation}

On the other hand, we then see that
\begin{equation}
|\phi_n(z_1)-\psi_n(z_1)|=\Big|\sum_{\zeta_j \neq (1,0)}\Big(\frac{r^\tau}{R(z, \zeta_j)-r^\tau}\Big)^{2n}\Big|.
\end{equation}

If $z_1$ is close enough to 1, the same argument as in the previous proof indicates that 
\begin{equation}
\Big|\sum_{\zeta_j \neq (1,0)}\Big(\frac{r^\tau}{R(z, \zeta_j)-r^\tau}\Big)^{2n}\Big| \leqslant \frac{A_2}{n}.
\end{equation}

Therefore, on this closed curve, we have 
\begin{equation}
|\phi_n(z_1)-\psi_n(z_1)| < |\psi_n(z_1)|,
\end{equation}
and then by Rouch\'{e}'s theorem, we know that $\phi_n$ also has at least a zero $\omega_{n,r}$ in the region bounded by the closed curve. 

According to the construction of the function $\phi_n$ we then know that $f_n$ has a zero $w_{n,r}=(\omega_{n,r},0)$. By checking the arguement again, we know that $w_{n,r}$ approaches to $(1,0)$ as $n$ goes to infinity. 

For other $\zeta_j$, we define
\begin{displaymath}
\phi_n(z_1)=f_n(z_1,z_2),
\end{displaymath}
and
\begin{displaymath}
\psi_n(z_1)=1-\varepsilon_n-\Big(\frac{r^\tau}{R\big((z_1,z_2), (1,0)\big)-r^\tau}\Big)^{2n},
\end{displaymath}
in both of which $z_2$ is such a complex number that $\pi(z_1,z_2)=\zeta_j$. Then we can do the same argument to show that the claim is true and in this case we have $\pi(w_{n,r})=\zeta_j$.
\end{proof}

The next key lemma states that if a boundary point is close enough to $\zeta_j$, then the broader approach region based there contains a zero of $f_n$.
\begin{lemma}
For each $n$ we can choose $r=r_n$ such that, if $w \in V_r(\zeta_j)$, then $\mathscr{A}_{h_1, h_2}(w)$ contains a zero of $f_n$.
\end{lemma}
\begin{proof}
Again, without loss of generality, we may assume that $\zeta_1=(1,0)$, and only check this case. For other situations, the same method applies.

Suppose $w_{n,r}=(\omega_{n,r}, 0)$ is the zero of $f_n$ near $(1,0)$ as we had in the previous lemma. So now our task is to verify that  if $|(1,0)-w| < Kr^\tau$ and $\Big|R\big(w,(1,0)\big)\Big|<\Lambda^{Kr}(1,0)$, we should have 
\begin{equation}
\left\{
\begin{array}{lr}
|w_{n,r}-w|< h_1\big(\delta_n(w_{n,r})\big) D(w_{n,r}),&\\
\Big|R\big(\pi(w_{n,r}),w\big)\Big|<h_2\big(\delta_n(w_{n,r})\big)\Lambda^{D(w_{n,r})}(w).& \\
\end{array} \right.
\end{equation}

Before starting the work, we need a foundation. We would like to claim that there exists a positive constant $C$ such that
\begin{equation}
\delta(w_{n,r}) > C r^\tau.
\end{equation}
If otherwise, we will have $\Big|R\big(\pi(w_{n,r}),(1,0)\big)-R\big(w_{n,r},(1,0)\big)\Big| < C_1 r^\tau$ for any constant $C_1$. However, this implies that $h(z)=R\big(z,(1,0)\big)$, written as a polynomial of $z$, does not have terms with degree less or equal to $\tau$ other than the constant term. This violates that the maximal type in $U$ is $\tau$.

First of all, we know that 
\begin{equation}
|w_{n,r}-w|<|(1,0)-w_{n,r}|+|(1,0)-w|.
\end{equation}

On the other hand, we check that
\begin{align}
h_1(\delta_n(w_{n,r})) D(w_{n,r}) &>h_1(\delta(w_{n,r})) D(w_{n,r}) \nonumber \\
&> K_3 h_1(\delta(w_{n,r})) \delta(w_{n,r}) \nonumber \\
                                 &> \frac{K_4}{2} h_1(\delta(w_{n,r}))|(1,0)-w_{n,r}|+ \frac{K_4}{2} h_1(\delta(w_{n,r}))\delta(w_{n,r}),
\end{align}
because $\delta(w_{n,r})=|(1,0)-w_{n,r}|$, as $\pi(w_{n,r})=(1,0)$. (Recall that in proof of the previous lemma we have the result that $\pi(w_{n,r})=\zeta_j$.)

If $r$ is small enough, it is true that
\begin{equation}
|(1,0)-w_{n,r}| < |(1,0)-w_{n,r}|\cdot \frac{K_4}{2} h_1(\delta(w_{n,r}))
\end{equation}
and 
\begin{equation}
|(1,0)-w| <K r^\tau <\frac{K}{C}\delta(w_{n,r})< \frac{K_4}{2} h_1(\delta(w_{n,r}))\delta(w_{n,r}).
\end{equation}

These imply that, if $r$ is small enough , we will have 
\begin{equation}
|w_{n,r}-w|< h_1(\delta_n(w_{n,r})) D(w_{n,r}).
\end{equation}

Meanwhile, we have
\begin{align}
\Big|R\big(\pi(w_{n,r}),w\big)\Big| & \leqslant \Big|R\big(\pi(w_{n,r}),w\big)-R\big(\pi(w_{n,r}),(1,0)\big)\Big| \nonumber \\
& \ \ + \Big|R\big(\pi(w_{n,r}),(1,0)\big)-R\big(w_{n,r}, (1,0)\big)\Big|+ \Big|R\big(w_{n,r},(1,0)\big)\Big| \nonumber \\
& <K_5 |w-(1,0)| + K_6 |\pi(w_{n,r})-w_{n,r}| + \Big|R\big((z_1,0),(1,0)\big)\Big| +\gamma_n r^\tau \nonumber \\
& <K_7 r^\tau + K_6 \delta(w_{n,r}) + \frac{\varepsilon_n}{2n}r^\tau + \gamma_n r^\tau +o(\frac{\varepsilon_n}{n}) \nonumber \\
& < K_8 \delta(w_{n,r}).
\end{align}

We then consider $h_2\big(\delta_n(w_{n,r})\big)\Lambda^{D(w_{n,r})}(w).$

By definition, we know that 
\begin{align}
h_2\big(\delta_n(w_{n,r})\big)\Lambda^{D(w_{n,r})}(w) &> K_9 h_2\big(\delta(w_{n,r})\big)\big(D(w_{n,r})\big)^{\tau_w} \nonumber \\
 & > K_{10} h_2\big(\delta(w_{n,r})\big) \big(\delta(w_{n,r})\big)^{\frac{\tau_w}{\tau_{w_{n,r}}}} \nonumber \\
 & > K_{11}h_2\big(\delta(w_{n,r})\big) \big(\delta(w_{n,r})\big)^{\frac{\tau}{2}}.
\end{align}

We can then find $r$ so small that this inequality holds:
\begin{equation}
K_8 \delta(w_{n,r}) < K_{11}h_2\big(\delta(w_{n,r})\big) \big(\delta(w_{n,r})\big)^{\frac{\tau}{2}}.
\end{equation}

In this way we check that 
\begin{equation}
\Big|R\big(\pi(w_{n,r}),w\big)\Big| < h_2\big(\delta_n(w_{n,r})\big)\Lambda^{D(w_{n,r})}(w).
\end{equation}

Therefore, there exists an $r>0$ such that 
\begin{equation}
\left\{
\begin{array}{lr}
|w_{n,r}-w|< h_1\big(\delta_n(w_{n,r})\big) D(w_{n,r}),&\\
\Big|R\big(\pi(w_{n,r}),w\big)\Big|<h_2\big(\delta_n(w_{n,r})\big)\Lambda^{D(w_{n,r})}(w),& \\
\end{array} \right.
\end{equation}
that is, $\mathscr{A}_{h_1, h_2}(w)$ contains a zero, $w_{n,r}$, of $f_{n,r}$.
\end{proof}

Now we are going to construct a bounded holomorphic function $f$ such that, for any $\zeta \in U$, the limit
\begin{displaymath}
\lim_{\mathscr{A}_{h_1, h_2}(\zeta) \ni z \to \zeta}f(z) 
\end{displaymath}
does not exist.

As we have seen from the proof above, for each $n$ we can choose $r_n$ such that the lemma is true. Then we choose a subsequence $\{\varepsilon_{n_k}\}$ such that $\sum{\varepsilon_{n_k}}<\infty$. Also for each $\zeta_j$, we can find a zero $w_{n_k}$ for $f_{n_k}$ such that $\{w_{n_k}\}$ converges to $\zeta_j$. 

We then build a bounded holomorphic function in $\Omega$:
\begin{displaymath}
f(z)=\prod_{k=1}^{\infty}\frac{f_{n_k}(z)}{1-(1-\varepsilon_{n_k})g_k(z)},
\end{displaymath}
where $g_k=g_{n_k, r_{n_k}}$.

Fix an arbitrary point $\zeta \in U$. 
We know that, for each $n_k$, with the corresponding number $r_{n_k}$, there is a maximal set $\{\zeta_j\}_{j \in J}$, and by the first lemma, there exists a point $\zeta_j$ such that $\zeta \in V_{ r_{n_k} }(\zeta_j)$. So, by the third lemma we know that a zero of $f_{n_k}(z)$, thus a zero of $f$, is contained in $\mathscr{A}_{h_1, h_2}(\zeta)$. As $k$ goes to infinity, $r_{n_k}$ converges to 0 and therefore $V_{r_{n_k}}(\zeta)$ is shrinking, making $\zeta_j$ converge to $\zeta$. Thus we have a sequence of zeros of $f$ that converges to the point $\zeta$. 

Suppose $\lim_{\mathscr{A}_{h_1, h_2}(\zeta) \ni z \to \zeta}f(z)$ exists, then so does  $\lim_{\mathscr{A}_{h_1, h_2}(\zeta) \ni z \to \zeta}|f(z)|$.

By evaluating $f$ along that sequence of zeros we have 
\begin{displaymath}
\liminf_{\mathscr{A}_{h_1, h_2}(\zeta) \ni z \to \zeta}|f(z)|=0,
\end{displaymath}
and therefore
\begin{displaymath}
\lim_{\mathscr{A}_{h_1, h_2}(\zeta) \ni z \to \zeta}|f(z)|=\limsup_{\mathscr{A}_{h_1, h_2}(\zeta) \ni z \to \zeta}|f(z)|=\liminf_{\mathscr{A}_{h_1, h_2}(\zeta) \ni z \to \zeta}|f(z)|=0,
\end{displaymath}
which means $|f|$, being a subharmonic function, is identically zero, because it can reach its supremum in an interior point. So we reach a contradiction, and therefore we know that 
\begin{displaymath}
\lim_{\mathscr{A}_{h_1, h_2}(\zeta) \ni z \to \zeta}f(z) 
\end{displaymath}
does not exist.

To summarize, we have the theorem
\begin{theorem}
\label{main}
On the boundary of $\Omega$ there exists a neighborhood $U$ of $(1,0)$, and there exists a bounded holomorphic function $f$, such that for any point $\zeta \in U$, the limit
\begin{displaymath}
\lim_{\mathscr{A}_{h_1, h_2}(\zeta) \ni z \to \zeta}f(z) 
\end{displaymath}
does not exist.
\end{theorem}

Since this $U$ is of positive measure, we immediately have
\begin{corollary}
\label{main2}
There is no Fatou's theorem for these broader approach regions $\mathscr{A}_{h_1, h_2}$.
\end{corollary}

\bibliographystyle{amsplain}

\end{document}